\documentclass[12pt]{amsart}

\usepackage{amsmath} 
\usepackage{amssymb} 
\usepackage{amsthm} 
\usepackage{dsfont} 
\usepackage{graphicx} 
\usepackage{enumerate} 
\usepackage{url} 

\newtheorem{theorem}{Theorem}[section]

\newtheorem{definition}[theorem]{Definition}
\newtheorem{example}[theorem]{Example}
\newtheorem{lemma}[theorem]{Lemma}

\begin{document}

\title[Unicity Concepts]{Unicity Concepts for Sudoku}

\author{Thomas Fischer}
\address{Fuchstanzstr. 20, 60489 Frankfurt am Main, Germany.}

\email{dr.thomas.fischer@gmx.de}

\date{}     


\begin{abstract}
This paper deals with a generalized Sudoku problem and investigates the unicity of a given 
solution. We introduce constraint sets, which is a generalization of the rows, columns and blocks 
of a classical Sudoku puzzle. The unicity property is characterized by three different properties.
We describe unicity by permutations, by unicity cells and by rectangles. These terms are defined 
in this paper and are illustrated with examples. Throughout this paper we are not concerned with 
the existence of a solution.
\end{abstract}

\keywords{Sudoku, Combinatorics, Unicity.}

\subjclass[2010]{Primary 90C35; Secondary 05B15 65K10}


\maketitle

                                        %
                                        %
\section{Introduction} \label{S:intro} 
The consideration of unicity properties is motivated by two perceptions. Usually, in a well-posed Sudoku 
it is required that there exists exactly one solution, i.e., there exists a solution and the solution is unique. 
This assumption makes it reasonable to investigate this property in order to gain additional insight, which 
can be used in the development of algorithms. On the other side in optimization and approximation theory 
it is a common approach to impose a unicity condition in order to obtain stronger results. A classical 
example is the alternation theorem for Haar subspaces in Chebyshev approximation. Haar subspaces 
are a global unicity assumption and the corresponding theory can be found, e.g., in the monograph of 
Cheney \cite{Che}.

This paper introduces three different types of unicity properties. In Section \ref{S:constraint_sets}
constraint sets are defined. These sets generalize the terms row, column and block of a classical 
Sudoku puzzle and prepare the unicity statements. In Section \ref{S:unicity_perm} we describe the 
solution set of the generalized Sudoku problem in terms of permutations. These permutations 
have specific properties on the constraint sets. The unique solvability depends on the nonexistence 
of such a permutation.

In Section \ref{S:unicity_cells} we introduce unicity cells, which is a local unicity condition. As a result
we prove that the generalized Sudoku problem is uniquely solvable if and only if each cell is a unicity 
cell. In Section \ref{S:unicity_rectangles} we consider a generalized type of rectangles as a set of 
cells. If the generalized Sudoku problem admits a unique solution with a minimality condition on a 
rectangle, then this rectangle contains a given, i.e., a populated cell.

Unicity properties for the Sudoku puzzle had been considered by Provan \cite{Pro} and Herzberg and 
Murty \cite{HM}. Provan showed that the complete solvability of a Sudoku puzzle using the pigeon-hole 
rule implies the unique solvability of the Sudoku puzzle. Herzberg and Murty presented a necessary
condition for the unique solvability.

Finally, we collect some basic terms and notations. Let $\mathds{Z}$ denote the set of integers. 
The $s$-times cartesian product of $\mathds{Z}$ is indicated by a superscript $s$, i.e.,
$\mathds{Z}^s$. The transpose of a vector or a matrix is indicated by a superscript $”T”$.
The vectors $\mathbf{0}$ respectively $\mathbf{1}$ denote the zero respectively one vector, 
consisting of zeros respectively ones in each component. The number of components is indicated 
by an index. Each vector is considered to be a column vector. $U$ denotes the identity matrix. 
The elements of a set are called distinct if each two elements of the set are distinct, i.e.,
if and only if the elements are pairwise distinct. The sum over an empty index set is considered 
to be zero. The symbol $\sharp$ denotes the number of elements (cardinality) of a finite set.

                                        %
                                        %
\section{The Mathematical Model} \label{S:model}
We treat Sudoku problems on the basis of a general model introduced in \cite{Fis1} and replicate the
model in this section. We restrict ourselves to the primal problem and do not consider the dual problem
introduced in \cite{Fis2}. Let $n$ be an integer with $n \ge 1$. We define the sum
\[                                                         
s(n)  = \sum_{i=1}^{n-1}i
\]
and define a matrix $A(n)$ with $s(n)$ rows and $n$ columns inductively. For $n=1$, let $A(1)$ 
denote the empty matrix, i.e., a matrix without entries. Assume the matrix $A(n-1)$ had been 
defined with $s(n-1)$ rows and $n-1$ columns. We set
{
\renewcommand{\arraystretch}{2.0}     
\[
A(n) = 
\begin{pmatrix}    
\begin{array}{c|c}
\mathbf{1}_{n-1} & - U_{n-1} \\
\hline
\mathbf{0}_{s(n-1)} & A(n-1)
\end{array}
\end{pmatrix}.
\]
}

We extend the matrix $A(n)$ to a matrix $A$ with $n \cdot s(n)$ rows and $n^2$ columns. 
The matrix $A$ consists in the ``main diagonal" of $n$ matrices $A(n)$ and the remaining values 
are set to zero. The matrix $A$ depends on the value $n$, but we do not state this 
dependence explicitly.

Given the set $\{1, \ldots, n^2\} \subset \mathds{Z}$, let $\pi$ be any permutation on this set, 
i.e.,
\[
\pi: \{1, \ldots, n^2\} \longrightarrow \{1, \ldots, n^2\}
\]
be a permutation. We extend the notion of permutation to the matrix $A$,  
i.e., we define $\pi (A) = (a^{\pi^{-1}(1)}, \ldots, a^{\pi^{-1}(n^2)})$, 
where $a^j$ denotes the $j^{th}$ column of $A$ for $j = 1, \ldots , n^2$. 
Given a permutation $\pi$ on $\{1, \ldots, n^2\}$, we define the matrix 
$A_\pi = \pi (A)$, i.e., we interchange the columns of $A$ according to 
the permutation $\pi$. 

\begin{definition} \label{D:nonzero}
Let $s \ge 1$. For any point $y = (y_1, \ldots, y_s)^T \in \mathds{Z}^s$, 
we write $y < > \mathbf{0}$ if each component of $y$ is nonzero, 
i.e., if $y_i \ne 0$ for $i=1, \ldots, s$.
\end{definition}

This definition should not be confused with the expression $y \ne \mathbf{0}$, 
where only one component of $y$ has to be nonzero.

Given is $n \ge 2$, some permutations $\pi_1$, $\pi_2$, $\pi_3$ on $\{1, \ldots, n^2\}$, 
some $0 \le k \le n^2$, an index set $\{i_1, \ldots, i_k \} \subset \{ 1, \ldots, n^2\}$, 
and givens $g_{i_1}, \ldots , g_{i_k} \in \mathds{Z}$ with $1 \le g_{i_l} \le n$ for 
$l = 1, \ldots, k$.

The generalized Sudoku problem is defined in the following way:
\begin{align*}
& \mbox{Find }x = (x_1, \ldots, x_{n^2})^T \in \mathds{Z}^{n^2}, \mbox{ such that }\\
& 1 \le x_i \le n \mbox{ for } i = 1, \ldots, n^2, \\
& A_{\pi_r} x <> \mathbf{0} \mbox{ for }r=1, 2, 3 \mbox{ and } \\
& x_{i_l} = g_{i_l} \mbox{ for }l = 1, \ldots, k.
\end{align*}

We restrict ourselves to this mathematical model and do not refer directly to the classical Sudoku 
puzzle. In particular, we will not investigate the relation of this model to the Sudoku puzzle
in detail. This had been described in \cite{Fis1} already.

                                        %
                                        %
\section{Constraint Sets} \label{S:constraint_sets}
From now on we denote the elements of  $\{1, \ldots, n^2\}$ in the definition of the generalized 
Sudoku problem as cells. We divide the set of all cells into $n$ subsets of $n$ cells. 

\begin{definition} \label{D:constraintsets}
Let $\pi$ be a permutation on $\{1, \ldots, n^2\}$. The $n$ sets 
$\{\pi((j-1)\cdot n + 1), \ldots, \pi(j \cdot n)\} = \{\pi(i) \mid (j-1)\cdot n + 1 \le i \le j \cdot n\}$
for $j = 1, \ldots, n$ are called the constraint sets of $\pi$ and are denoted by 
$cs_\pi(1), \ldots, cs_\pi(n)$.
\end{definition}

Each of the sets $cs_\pi(j)$ contains exactly $n$ elements. If the generalized Sudoku problem
describes a classical Sudoku puzzle (as defined in \cite[Section 3]{Fis1}), then the constraint 
sets of  $\pi_1$, $\pi_2$ respectively $\pi_3$ describe the indices of the rows, columns 
respectively blocks of the Sudoku square. 

\begin{lemma} \label{L:31}
Let $\pi$ be a permutation on $\{1, \ldots, n^2\}$. The constraint sets of $\pi$ form 
a partition of the cells $\{1, \ldots, n^2\}$, i.e., 
\begin{enumerate}[(i)] 
\item $cs_\pi(i) \cap cs_\pi(j) = \emptyset$ for $i, j = 1, \ldots, n, i \ne j$, and 
\item $\bigcup^n_{j=1} cs_\pi(j) = \{1, \ldots, n^2\}$.
\end{enumerate}
\end{lemma}
\begin{proof}
(i) This is clear, because $\pi$ is a one-to-one mapping.  \\
(ii) Obviously, $cs_\pi(j) \subset  \{1, \ldots, n^2\}$ for $j=1, \ldots, n$. The inclusion 
$\{1, \ldots, n^2\} \subset  \bigcup^n_{j=1} cs_\pi(j)$ follows, because $\pi$ is onto.
\end{proof}

Let $\pi$ be a permutation on  $\{1, \ldots, n^2\}$  and let $x = (x_1, \ldots, x_{n^2})^T
\in \mathds{Z}^{n^2}$. 
We adopt a notation of \cite[Section 2]{Fis1} and define
\[
\pi(x) = (x_{\pi^{-1}(1)}, \ldots, x_{\pi^{-1}(n^2)})^T.
\]

The proof of the next lemma follows immediately from \cite[Lemma 3.2]{Fis1}.

\begin{lemma} \label{L:32}
Let $\pi$ be a permutation on $\{1, \ldots, n^2\}$, let $j \in \{1, \ldots, n\}$ and let 
$x = (x_1, \ldots, x_{n^2})^T \in \mathds{Z}^{n^2}$. 
The following statements are equivalent: 
\begin{enumerate}[(i)] 
\item $A(n) (x_{\pi((j-1)n+1)}, \ldots, x_{\pi(j \cdot n)}) < > \mathbf{0}$.  
\item $x_{\pi((j-1)n+1)}, \ldots, x_{\pi(j \cdot n)}$ are distinct.
\end{enumerate}
\end{lemma}

From \cite[Lemma 3.3]{Fis1}, we derive this lemma.

\begin{lemma} \label{L:33}
Let $\pi$ be a permutation on $\{1, \ldots, n^2\}$, let $j \in \{1, \ldots, n\}$ and let 
$x = (x_1, \ldots, x_{n^2})^T \in \mathds{Z}^{n^2}$. 
The following statements are equivalent: 
\begin{enumerate}[(i)] 
\item $A(n) (x_{\pi((j-1)n+1)}, \ldots, x_{\pi(j \cdot n)}) < > \mathbf{0}$  \\
and $1 \le x_{\pi((j-1)n+i)}  \le n$ for $i=1, \ldots, n$.  
\item $\{x_i \mid i \in cs_\pi(j) \} = \{ 1, \ldots, n \}$.
\end{enumerate}
\end{lemma}

The results on the subvectors of $x$ can be extended to the full length.

\begin{lemma} \label{L:34}
Let $\pi$ be a permutation on $\{1, \ldots, n^2\}$ and let 
$x = (x_1, \ldots, x_{n^2})^T \in \mathds{Z}^{n^2}$. 
The following statements are equivalent: 
\begin{enumerate}[(i)] 
\item $A_\pi x < > \mathbf{0}$.  
\item $x_{\pi((j-1)n+1)}, \ldots, x_{\pi(j \cdot n)}$ are distinct for $j = 1, \ldots, n$.
\end{enumerate}
\end{lemma}
\begin{proof}
Using \cite[Lemma 3.4 (i)]{Fis1},
\[
A_\pi x = A \pi^{-1}(x) = A (x_{\pi(1)}, \ldots, x_{\pi(n^2)})^T.
\]
The claim follows from the diagonal structure of $A$ and Lemma \ref{L:32}.
\end{proof}

From Lemma \ref{L:34} and \cite[Lemma 3.1]{Fis1}, we derive immediately the next lemma.

\begin{lemma} \label{L:35}
Let $\pi$ be a permutation on $\{1, \ldots, n^2\}$ and let 
$x = (x_1, \ldots, x_{n^2})^T \in \mathds{Z}^{n^2}$. 
The following statements are equivalent: 
\begin{enumerate}[(i)] 
\item $A_\pi x < > \mathbf{0}$ and $1 \le x_i \le n$ for $i = 1, \ldots, n^2$.  
\item $\{ x_i \mid i \in cs_\pi(j) \} = \{1, \ldots, n \}$ for $j = 1, \ldots, n$.
\end{enumerate}
\end{lemma}

In the sequel we use from time to time the notation $cs_\pi$ respectively $cs$ for constraint sets,
i.e., we do not number them, but use an abbreviated version.

                                        %
                                        %
\section{Unicity by Permutations} \label{S:unicity_perm}
In this section, we consider the unique solvability of the generalized Sudoku problem
introduced in Section \ref{S:model} and we focus on a description of unicity using the 
dependance on permutations of $\{1, \ldots, n^2\}$.

\begin{definition} \label{D:pi-consistent}
Let $\pi$ be a permutation on $\{1, \ldots, n^2\}$. A permutation $\tau$ on
$\{1, \ldots, n^2\}$ is called $\pi$-consistent if $\tau(cs_\pi(j)) = cs_\pi(j)$ for 
$j = 1, \ldots, n$.
\end{definition}

A permutation is called $\pi$-consistent if it preserves the constraint sets of $\pi$.
In a classical Sudoku puzzle this definition states that $\tau$ permutes only values within 
a row, a column or a block (depending on $\pi$). 

Please note, $\tau$ is $\pi$-consistent if and only if $\tau^{-1}$ is $\pi$-consistent. 
The concatenation $\pi \circ \tau^{-1}$ of two permutations $\pi$ and $\tau^{-1}$ 
is a permutation again.

\begin{lemma} \label{L:41}
Let $\pi$ and $\tau$ be permutations on $\{1, \ldots, n^2\}$ and let 
$x = (x_1, \ldots, x_{n^2})^T \in \mathds{Z}^{n^2}$. The following 
statements are equivalent: 
\begin{enumerate}[(i)] 
\item $A_\pi \tau(x) <> \mathbf{0}$. 
\item $x_{(\pi\circ\tau^{-1})((j-1) n+1)}, \ldots, x_{(\pi\circ\tau^{-1})(j \cdot n)}$ 
are distinct for $j = 1, \ldots, n.$ 
\end{enumerate}
\end{lemma}
\begin{proof}
We apply \cite[Lemma 3.4 (i)]{Fis1} and obtain
\[
A_\pi \tau(x) = A\pi^{-1}(\tau(x)) = A(\tau \circ \pi^{-1})(x) = A_{\pi\circ\tau^{-1}}x.
\]
Using Lemma \ref{L:34} with $\pi \circ \tau^{-1}$ (in the role of $\pi$) gives the desired result.
\end{proof}

\begin{lemma} \label{L:42}
Let $\pi$ be a permutation on $\{1, \ldots, n^2\}$, let $\tau$ be a $\pi$-consistent 
permutation on $\{1, \ldots, n^2\}$ and let $x = (x_1, \ldots, x_{n^2})^T
\in \mathds{Z}^{n^2}$. The following statements hold: 
\begin{enumerate}[(i)] 
\item $1 \le x_i \le n$ for $i = 1, \ldots, n^2$ if and only if $1 \le x_{\tau^{-1}(i)} \le n$  \\
for $i = 1, \ldots, n^2$.  
\item $A_\pi x <> \mathbf{0}$ if and only if $A_\pi \tau(x) <> \mathbf{0}$.
\end{enumerate}
\end{lemma}
\begin{proof}
``(i)" This is clear, because $\tau$ is a permutation and does not modify the values of the 
components of $x$. \\
``(ii)" Using Lemma \ref{L:34}, $A_\pi x < > \mathbf{0}$ if and only if 
\[
x_{\pi((j-1) n+1)}, \ldots, x_{\pi(j \cdot n)} \mbox{ are distinct for }j = 1, \ldots, n. 
\]
By assumption, $\tau$ respectively $\tau^{-1}$ are $\pi$-consistent, i.e., $A_\pi x <> \mathbf{0}$ 
if and only if
\[
x_{\tau^{-1}(\pi((j-1) n+1))}, \ldots, x_{\tau^{-1}(\pi(j \cdot n))}
\mbox{ are distinct for }j = 1, \ldots, n. 
\]
The claim follows from Lemma \ref{L:41}.
\end{proof}

\begin{definition} \label{D:pi-x-consistent}
Let $\pi$ be a permutation on $\{1, \ldots, n^2\}$ and let 
$x = (x_1, \ldots, x_{n^2})^T \in \mathds{Z}^{n^2}$. A permutation 
$\tau$ on $\{1, \ldots, n^2\}$ is called $\pi$-$x$-consistent if 
\[
\{ x_{\tau^{-1}(i)} \mid i \in cs_\pi(j) \} = \{ x_i \mid i \in cs_\pi(j) \} 
\]
for $j = 1, \ldots, n$.
\end{definition}

In a classical Sudoku puzzle this definition states, that $\tau$ may permute not only values 
within a row, a column or a block, but more changes are allowed. It is required, that after 
the permutation each row, column and block contains the same values as before. Each 
$\pi$-consistent permutation is $\pi$-$x$-consistent for each 
$x \in \mathds{Z}^{n^2}$ and each permutation $\pi$ on $\{1, \ldots, n^2\}$.

\begin{lemma} \label{L:43}
Let $\pi$ and $\tau$ be permutations on $\{1, \ldots, n^2\}$ and let 
$x = (x_1, \ldots, x_{n^2})^T \in \mathds{Z}^{n^2}$, such that $1 \le x_i \le n$ 
for $i = 1, \ldots, n^2$. Consider the statements: 
\begin{enumerate}[(i)] 
\item $A_\pi x <> \mathbf{0}$. 
\item $A_\pi \tau(x) <> \mathbf{0}$. 
\item $\tau$ is $\pi$-$x$-consistent.
\end{enumerate}
Each two of the statements imply the third.
\end{lemma}
\begin{proof}
By Lemma \ref{L:35}, (i) holds if and only if
\[
\{ x_i \mid i \in cs_\pi(j) \} = \{1, \ldots, n \} \mbox{ for }j = 1, \ldots, n.
\]
Using Lemma \ref{L:41} and \cite[Lemma 3.1]{Fis1}, (ii) holds if and only if
\[
\{ x_i \mid i \in cs_{\pi\circ\tau^{-1}}(j) \} = \{1, \ldots, n \} \mbox{ for }j = 1, \ldots, n.
\]
Substituting $i$ with $\tau^{-1}(i)$,
\begin{align*}
\{ x_i \mid i \in cs_{\pi\circ\tau^{-1}}(j) \} 
& = \{ x_{\tau^{-1}(i)} \mid i \in \tau(cs_{\pi \circ \tau^{-1}}(j)) \} \\
& = \{ x_{\tau^{-1}(i)} \mid i \in cs_\pi(j) \} 
\end{align*}
for $j = 1, \ldots, n$, i.e., (ii) holds if and only if 
\[
\{ x_{\tau^{-1}(i)} \mid i \in cs_\pi(j) \}  = \{1, \ldots, n \} \mbox{ for }j = 1, \ldots, n.
\]
The claim follows from the definition of  $\pi$-$x$-consistent.
\end{proof}

We define the solution set of the generalized Sudoku problem by
\begin{align*}
S(n, g) 
= \{x = (x_1, \ldots, x_{n^2})^T \in \mathds{Z}^{n^2} \mid 
& \: 1 \le x_i \le n \mbox{ for } i = 1, \ldots, n^2, \\
& \: A_{\pi_r} x <> \mathbf{0} \mbox{ for } r = 1, 2, 3 \mbox{ and } \\
& \: x_{i_l} = g_{i_l}  \mbox{ for }l = 1, \ldots, k \: \}.
\end{align*}

\begin{lemma} \label{L:44}
Let $x, y \in S(n, g)$ and let $r \in \{1, 2, 3\}$. There exists a permutation $\tau$ 
on $\{1, \ldots, n^2\}$ with the following properties: 
\begin{enumerate}[(i)] 
\item $y = \tau(x)$, 
\item $\tau$ is $\pi_r$-consistent, 
\item $\tau$ is $\pi_s$-$x$-consistent for $s = 1, 2, 3$ and 
\item $\tau(i_l) = i_l$ for $l=1, \ldots, k$.
\end{enumerate}
\end{lemma}
\begin{proof}
Let $x = (x_1, \ldots, x_{n^2})^T$ and $y = (y_1, \ldots, y_{n^2})^T$.
We apply Lemma \ref{L:35} to $\pi_r$, $x$ and $y$, which implies
\[
\{ x_i \mid i \in cs_{\pi_r}(j) \} = \{ y_i \mid i \in cs_{\pi_r}(j) \} = \{1, \ldots, n\}
\]
for $j = 1, \ldots, n$. There exist permutations $\tau_j$ on $cs_{\pi_r}(j)$, 
such that $x_i = y_{\tau_j(i)}$ for each $i \in cs_{\pi_r}(j)$ and $j=1, \ldots, n$.
By Lemma \ref{L:31}, the constraint sets form a partition of $\{1, \ldots, n^2\}$ 
and we can combine the $\tau_j$ to a permutation $\tau$ on $\{1, \ldots, n^2\}$ with
\[
\tau_{\mid cs_{\pi_r}(j)} = \tau_j \mbox{ for }j=1, \ldots, n. 
\]

Here $\mid$ denotes the restriction of a mapping to a set.
This $\tau$ is also a permutation, since all $\tau_j's$ are permutations and the 
constraint sets form a partition of $\{1, \ldots, n^2\}$.

``(i)" Let $i \in \{1, \ldots, n^2\}$ be any cell. There exists a (uniquely determined) 
index $1 \le j \le n$, such that  $i \in cs_{\pi_r}(j)$, i.e., $i$ is contained in the $j^{th}$ 
constraint set. This implies $y_{\tau(i)} = y_{\tau_j(i)} = x_i$. Since $i$ had been chosen 
arbitrarily, we obtain $x_{\tau^{-1}(i)} = y_i$ for $i \in \{1, \ldots, n^2\}$ and
\[
\tau(x) = (x_{\tau^{-1}(1)}, \ldots, x_{\tau^{-1}(n^2)}) 
= (y_1, \ldots, y_{n^2}) = y, \mbox{ i.e., (i)}.
\]

``(ii)" Obviously, $\tau$ is $\pi_r$-consistent, since each $\tau_j$ is a permutation on 
$cs_{\pi_r}(j)$ for $j=1, \ldots, n$.

``(iii)" Let $s \in \{1, 2, 3\}$. Since $y \in S(n, g)$,
\[
A_{\pi_s}\tau(x) = A_{\pi_s} y <> \mathbf{0}
\]
and, by Lemma \ref{L:43}, $\tau$ is $\pi_s$-$x$-consistent.

``(iv)" Let $l \in \{1, \ldots, k\}$. There exists an index $1 \le j \le n$, such that 
$i_l \in cs_{\pi_r}(j)$. Then $\tau_j(i_l) \in cs_{\pi_r}(j)$ and
\[
y_{i_l} = g_{i_l} = x_{i_l} = y_{\tau_j(i_l)} = y_{\tau(i_l)}.
\]
By Lemma \ref{L:34}, $y_s \ne y_t$ for each $s, t \in cs_{\pi_r}(j)$ if $s \ne t$. 
This implies $i_l = \tau_j(i_l) = \tau(i_l)$ and proves the claim.
\end{proof}

Applied to a classical Sudoku puzzle with $r=1$ the permutation $\tau$ describes a
mapping, which changes cells within a row. This permutation does not preserve 
the columns, but after the permutation we have distinct values in each column. 
The cells with a given are not touched by the permutation $\tau$.

Now we prove two descriptions of the solution set of the generalized Sudoku 
problem. These descriptions are based on permutations.

\begin{theorem} \label{T:41}
Let $x$ be a solution of the generalized Sudoku problem, i.e., $x \in S(n, g)$. Then
\begin{align*}
S(n, g) = \{y \in \mathds{Z}^{n^2} \mid 
& \: \mbox{there exist permutations } \tau_1, \tau_2 \mbox{ and } \tau_3 \\
& \: \mbox{on } \{1, \ldots, n^2\}, \mbox{such that} \\
& \: y = \tau_r(x) \mbox{ for }r = 1, 2, 3, \\
& \: \tau_r \mbox{ is }\pi_r\mbox{-consistent for }  r = 1, 2, 3, \mbox{and} \\
& \: \tau_r(i_l) = i_l \mbox{ for }l=1, \ldots, k \mbox{ and } r = 1, 2, 3 \: \}.
\end{align*}
\end{theorem}
\begin{proof}
``$\subset$" Let $y \in S(n, g)$. By Lemma \ref{L:44}, there exist permutations 
$\tau_1$, $\tau_2$ and $\tau_3$ on $\{1, \ldots, n^2\}$, such that $y = \tau_r(x)$ for 
$r = 1, 2, 3$, $\tau_r$ is $\pi_r$-consistent  for $r = 1, 2, 3$ and
$\tau_r(i_l) = i_l$  for $l=1, \ldots, k$ and $r=1, 2, 3$. \\
``$\supset$" Let  $y \in \mathds{Z}^{n^2}$, let $\tau_1$, $\tau_2$ and $\tau_3$ 
be permutations on $\{1, \ldots, n^2\}$, such that $y = \tau_r(x)$ for $r = 1, 2, 3$,
$\tau_r$ is $\pi_r$-consistent for $r = 1, 2, 3$ and $\tau_r(i_l) = i_l$  for $l=1, \ldots, k$ 
and $r=1, 2, 3$. Let $r \in \{1, 2, 3\}$ and $j \in \{1,\ldots, n\}$, then
\begin{align*}
\{ y_i \mid i \in cs_{\pi_r}(j) \} 
& = \{ x_{\tau^{-1}_r(i)} \mid i \in cs_{\pi_r}(j) \} \\
& = \{ x_i \mid i \in cs_ {\pi_r}(j) \} \\
& = \{1, \ldots, n\},
\end{align*}
since $\tau_r$ is $\pi_r$-consistent and using Lemma \ref{L:35}. By Lemma \ref{L:35} applied 
to $y$, $A_{\pi_r}y <> \mathbf{0}$ and $1 \le y_i \le n$ for $i=1, \ldots, n^2$. Moreover 
$y_{i_l} = x_{\tau^{-1}_1(i_l)} = x_{i_l} = g_{i_l}$ for $l=1, \ldots, k$, i.e., $y \in S(n, g)$.
\end{proof}

For each solution $y \in S(n, g)$ the permutations in Theorem \ref{T:41} are uniquely determined 
as we see from the next lemma.

\begin{lemma} \label{L:45}
Let $\pi$ be a permutation on $\{1, \ldots, n^2\}$, let $x \in \mathds{Z}^{n^2}$ 
such that $A_\pi x <> \mathbf{0}$ and let $y \in \mathds{Z}^{n^2}$. Let $\tau_1$ and 
$\tau_2$ be $\pi$-consistent permutations on $\{1, \ldots, n^2\}$, such that $y=\tau_1(x)$ 
and $y=\tau_2(x)$. Then $\tau_1 = \tau_2$.
\end{lemma}
\begin{proof}
Let $i \in \{1, \ldots, n^2\}$. By Lemma \ref{L:31}, there exists $j \in \{1, \ldots, n\}$ and a 
constraint set $cs_\pi(j)$, such that $i \in  cs_\pi(j)$. By assumption, $\tau_1(x) = \tau_2(x)$, hence 
$x_{\tau_1^{-1}(i)} = x_{\tau_2^{-1}(i)}$. Since $\tau_1$ and $\tau_2$ are $\pi$-consistent, 
$\tau_1^{-1}(i) \in cs_\pi(j)$ and $\tau_2^{-1}(i) \in cs_\pi(j)$. By Lemma \ref{L:34},
$\tau_1^{-1}(i) = \tau_2^{-1}(i)$. Since $i$ had been chosen arbitrarily, $\tau_1^{-1} = \tau_2^{-1}$ 
and we obtain
$\tau_1 = (\tau_2 \circ \tau_2^{-1}) \circ \tau_1 = \tau_2 \circ (\tau_1^{-1} \circ \tau_1) = \tau_2$.
\end{proof}

The description in Theorem \ref{T:41} uses the condition $\pi$-consistent for three permutations. 
The next theorem describes the solution set by one permutation with the weaker condition 
$\pi$-$x$-consistent. Again by Lemma \ref{L:45}, the permutation $\tau$ in Theorem \ref{T:42} 
is also uniquely determined for each solution $y$.

\begin{theorem} \label{T:42}
Let $x$ be a solution of the generalized Sudoku problem, i.e., $x \in S(n, g)$. Then
\begin{align*}
S(n, g) 
= \{y \in \mathds{Z}^{n^2} \mid 
& \: \mbox{there exists a permutation }\tau \mbox{ on } \{1, \ldots, n^2\}, \\
& \: \mbox{such that } y = \tau(x), \tau \mbox{ is } \pi_1\mbox{-consistent }, \\
& \: \pi_2\mbox{-}x\mbox{-consistent}, \pi_3\mbox{-}x\mbox{-consistent} \mbox{ and} \\
& \: \tau(i_l) = i_l \mbox{ for }l=1, \ldots, k \: \}.
\end{align*}
\end{theorem}
\begin{proof}
``$\subset$"  Let $y \in S(n, g)$. Using Lemma \ref{L:44}, there exists a permutation $\tau$ 
with the desired properties. \\
``$\supset$" Let $y \in \mathds{Z}^{n^2}$, $\tau$ be a permutation, such that $y=\tau(x)$,
$\tau$ is $\pi_1$-consistent, $\pi_2$-$x$-consistent, $\pi_3$-$x$-consistent and $\tau(i_l) = i_l$ 
for $l=1, \ldots, k$. The components of $y$ are permutations of the components of $x$, i.e., 
$1 \le y_i \le n$ for $i=1, \ldots, n^2$. By Lemma \ref{L:43},
\[
A_{\pi_r} y = A_{\pi_r}\tau(x) <> \mathbf{0} \mbox{ for }r=1, 2, 3.
\]
Moreover $y_{i_l} = x_{\tau^{-1}(i_l)} = x_{i_l} = g_ {i_l}$ for $l=1, \ldots, k$
and this shows $y \in S(n, g)$.
\end{proof}

From this theorem we can derive a characterization of uniquely solvable generalized 
Sudoku problems.

\begin{theorem} \label{T:43}
Let $x$ be a solution of the generalized Sudoku problem, i.e., $x \in S(n, g)$. Then $x$ is the
unique solution of the generalized Sudoku problem if and only if there does not exist a 
permutation $\tau$ on $\{1, \ldots, n^2\}$, such that $\tau(x) \ne x$, $\tau$ is 
$\pi_1$-consistent, $\pi_2$-$x$-consistent, $\pi_3$-$x$-consistent and $\tau(i_l) = i_l$ for 
$l=1, \ldots, k$.
\end{theorem}

                                        %
                                        %
\section{Unicity by Unicity Cells} \label{S:unicity_cells}
In this section, we consider a unicity concept, which we call unicity cell. This definition allows ``local" 
considerations of unicity, i.e., we characterize, when a single cell contains a unique value. 

\begin{definition} \label{D:unicity_cell}
A cell $i \in \{1, \ldots, n^2\}$ is called a unicity cell with unique value $v \in \{1, \ldots, n\}$ 
if all solutions $x=(x_1, \ldots, x_{n^2})^T$ of the generalized Sudoku problem satisfy 
$x_i = v$.
\end{definition}

Please note, this definition does not say anything on the solvability of the generalized Sudoku 
problem. In an unsolvable generalized Sudoku problem each cell is a unicity cell for trivial reasons.
It is immediately clear from the definition, that a solvable generalized Sudoku problem admits a 
unique solution if and only if each cell $i \in \{1, \ldots, n^2\}$ is a unicity cell.

We continue with the definition of a Sudoku subproblem, which we call the reduced problem.
The condition $x_{i_l} = g_{i_l}$ for $l = 1, \ldots, k$ in the definition of the generalized Sudoku 
problem can be written as $A_{eq} x = g$ with a suitable $k \times n^2$ matrix $A_{eq}$ 
(describing the equality conditions) and a right side 
$g = (g_{i_1}, \ldots, g_{i_k})^T \in \mathds{Z}^k$. Roughly spoken, the matrix $A_{eq}$ is a 
unit matrix, where some of the rows had been left. 

We divide the set of all cells $\{1, \ldots, n^2\}$ into the set of known cells 
$\{i_1, \ldots, i_k\}$ (determining $A_{eq}$ and $g$) and the set of unknown cells 
$\{1, \ldots, n^2\} \backslash \{i_1, \ldots, i_k\}$ (see Section \ref{S:model}). Each known cell $i_l$ is a 
unicity cell with unique value $g_{i_l}$ for $l=1, \ldots, k$.

Let $1 \le p \le n^2$ be an integer and let $J=\{j_1, \ldots, j_p\} \subset 
\{1, \ldots, n^2\}$ be a set of cells with $j_1 < \ldots < j_p$ (and possibly containing known 
and unknown cells). We consider the projection
 $P_J: \mathds{Z}^{n^2} \longrightarrow \mathds{Z}^p$
defined by 
$P_J(x) = (x_{j_1}, \ldots, x_{j_p})^T \mbox{ for each }
x=(x_1, \ldots, x_{n^2})^T \in \mathds{Z}^{n^2}.$

This projection induces a transformation $t_J: J \longrightarrow \{1, \ldots, p\}$ defined by 
$t_J(j_l) = l$ for $l=1, \ldots, p$. The $t_J(j)^{th}$-component of $P_J(x)$, $P_J(x)_{t_J(j)}=x_j$ 
for each $x=(x_1, \ldots, x_{n^2})^T \in \mathds{Z}^{n^2}$ and $j \in J$.

For any $q \times n^2$-matrix 
\[
M=
\begin{pmatrix} 
m_1^T \\
\vdots \\
m_q^T
\end{pmatrix}
\]
with $q\in\mathds{Z}$ rows, $q \ge 1$, and vectors $m_l \in \mathds{Z}^{n^2}$ for $l=1, \ldots, q$,
we define 
\[
P_J(M)=
\begin{pmatrix} 
P_J(m_1)^T \\
\vdots \\
P_J(m_q)^T
\end{pmatrix}.
\]
We consider the projected matrices $P_J(A_{\pi_1})$, $P_J(A_{\pi_2})$, $P_J(A_{\pi_3})$ and 
$P_J(A_{eq})$ and delete all zero rows. Additionally, we delete all rows from $P_J(A_{\pi_r})$ 
for $r=1, 2, 3$ with at most one nonzero component. The resulting matrices are denoted by $B_1$, $B_2$, 
$B_3$ and $B_{eq}$. We also delete the components of the point $g=(g_{i_1}, \ldots, g_{i_k})^T \in
\mathds{Z}^k$, which belong to deleted rows of $P_J(A_{eq})$ and denote the new point by $g^\prime$.
We define the reduced problem induced by $J$ in the following way:
\begin{align*}
& \mbox{Find }z = (z_1, \ldots, z_p)^T \in \mathds{Z}^p, \\
& \mbox{such that }1 \le z_i \le n \mbox{ for } i = 1, \ldots, p, \\
& B_r z <> \mathbf{0} \mbox{ for }r=1, 2, 3 \mbox{ and } B_{eq} z = g^\prime.
\end{align*}

The reduced problem is built from the generalized Sudoku problem by dropping cells. Simultaneously,
all comparisons between dropped cells and between remaining cells and dropped cells are also dropped.
All givens in the dropped cells do not appear in the reduced problem. The generalized Sudoku 
problem and the reduced problem induced by $J$ are related.

\begin{lemma} \label{L:51}
(i) Let $x$ be a solution of the generalized Sudoku problem and let $J \subset \{1, \ldots, n^2\}$ 
be a set of cells. Then $P_J(x)$ is a solution of the reduced problem induced by $J$. \\
(ii) If $J=\{1, \ldots, n^2\}$, then the reduced problem induced by $J$ equals the 
generalized Sudoku problem.
\end{lemma}
\begin{proof}
``(i)" This is clear, since all constraints of the reduced problem induced by $J$ are also constraints 
of the generalized Sudoku problem. \\
``(ii)" If $J=\{1, \ldots, n^2\}$, then $p=n^2$ and $P_J$ is the identical mapping.
Both problems are identical.
\end{proof}

Obviously, the converse of Lemma \ref{L:51} (i) is not true. If $P_J(x)$ is a solution of the 
reduced problem, there is no guarantee, that $x$ satisfies all constraints of the generalized 
Sudoku problem. 

There is an immediate consequence of Lemma \ref{L:51}. If the generalized Sudoku problem 
is solvable, then the reduced problem is also solvable. 

\begin{definition} \label{D:unicity_cell_wrt}
Let $1 \le p \le n^2$ be an integer, let $J \subset \{1, \ldots, n^2\}$ be a set of $p$ cells, let $i \in J$ 
and let $v \in \{1, \ldots, n\}$. The cell $i$ is called a unicity cell w.r.t. the base set $J$ with unique 
value $v$ if all solutions $x=(x_1, \ldots, x_p)^T$ of the reduced problem induced by $J$ satisfy 
$x_{t_J(i)} = v$.
\end{definition}

Note, the conditions in Definition \ref{D:unicity_cell_wrt} do not require the solvability of the reduced 
problem. We also do not impose any minimality condition on the base set. 

Using Lemma \ref{L:51} (ii), a cell $i$ is a unicity cell with unique value $v$ (see Definition 
\ref{D:unicity_cell}) if and only if $i$ is a unicity cell w.r.t. the base set $\{1, \ldots, n^2\}$ with unique 
value $v$. 

Mainly, we are interested in ``small" base sets $J$, i.e., with a small number of elements. A possible 
solution method could try to determine a base set in each step and to solve the reduced problem 
induced by it.

\begin{example} \label{E:51}
We consider the example depicted in Fig. \ref{Fig1} and the union $J$ of the sets 
$\{1, 2, 3, 10, 11, 12, 19, 20, 21 \}$ (the first block) and $\{4, 16, 28, 56\}$ 
(the cells populated with $4$'s). The cell $21$ (row 3, column 3) is a unicity cell w.r.t. $J$
with unique value $4$. All solutions of the reduced problem induced by $J$ 
contain a $4$ in the $11^{th}$ component (which corresponds to cell $21$).
\begin{figure}
\includegraphics{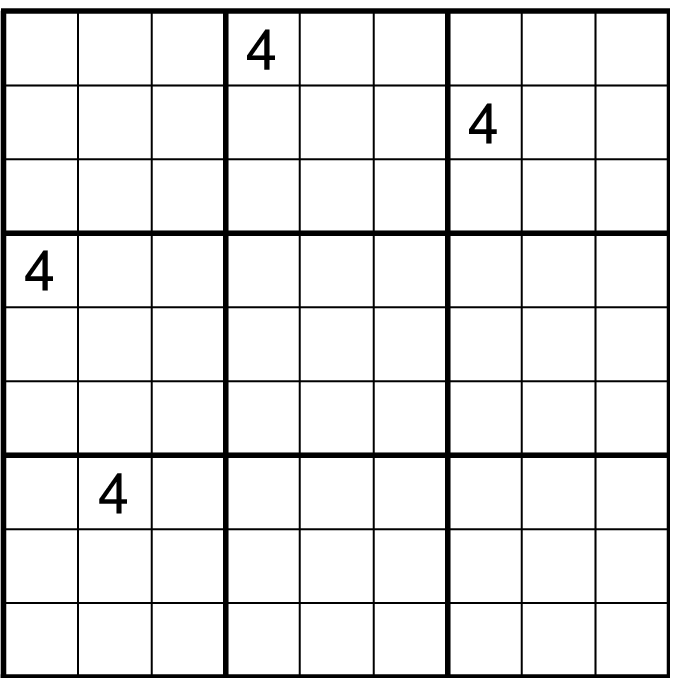}
\caption{Sudoku with a unicity cell} 
\label{Fig1}
\end{figure}
\end{example}

The relation between unicity cell and unicity cell w.r.t. a base set will be examined in the following theorem.

\begin{theorem} \label{T:51}
Let $i \in \{1, \ldots, n^2\}$ be a cell and let $v \in \{1, \ldots, n\}$. The following 
statements are equivalent: 
\begin{enumerate}[(i)] 
\item The cell $i$ is a unicity cell with unique value $v$. 
\item There exists a base set $J\subset\{1, \ldots, n^2\}$, such that $i\in J$  \\
and $i$ is a unicity cell w.r.t. $J$ with unique value $v$.
\end{enumerate}
\end{theorem}
\begin{proof}
``(i)$\Rightarrow$(ii)" Define $J=\{1, \ldots, n^2\}$. The claim follows from Lemma 
\ref{L:51} (ii). \\
``(ii)$\Rightarrow$(i)" Let $x=(x_1, \ldots, x_{n^2})^T$ and $y=(y_1, \ldots, y_{n^2})^T$ 
be solutions of the generalized Sudoku problem. By assumption (ii), there exists a base set 
$J\subset \{1, \ldots, n^2\}$, such that $i \in J$ and $i$ is a unicity cell w.r.t. $J$ with unique 
value $v$. We consider the reduced problem induced by $J$ and by Lemma \ref{L:51} (i), 
$P_J(x)$ and $P_J(y)$ are solutions of the reduced problem. Consequently, 
$x_i = P_J(x)_{t_J(i)} = v = P_J(y)_{t_J(i)}=y_i$.
\end{proof}

The strategy of humans to solve a uniquely solvable Sudoku is to find a cell $i$ and a ``pretty small" 
base set $J$. The preceeding theorem suggests, it is not necessary to find some cell $i$, but we can 
choose any cell $i$ and we still are able to find a base set $J$. But the price for selecting an arbitrary 
cell $i$ may be a large base set $J$, which results in a difficult to solve reduced problem. We illustrate 
this with the next example. In the next two examples we change the indexation to a more readable 
form $x_{row, column}$. 

\begin{example} \label{E:52}
\begin{figure}
\includegraphics{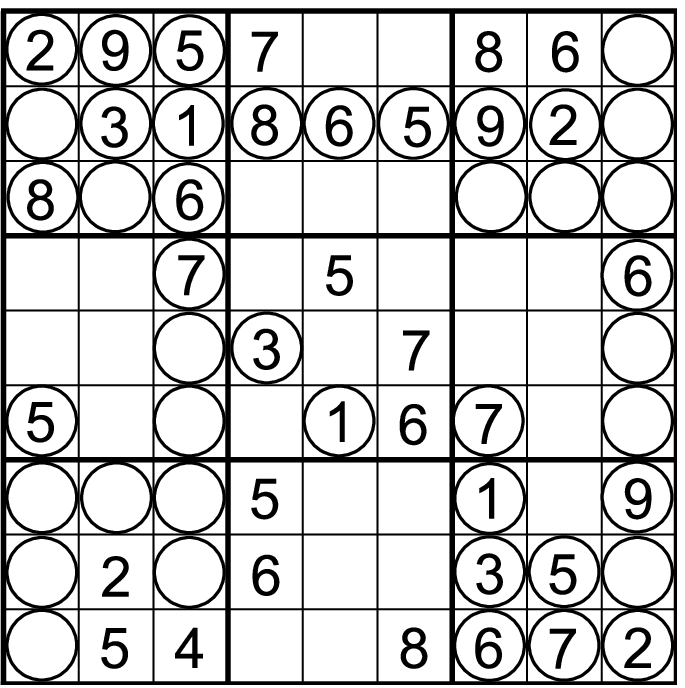}
\caption{The example of Mepham} 
\label{Fig2}
\end{figure}
We consider a continuation of the ``diabolical" example of Mepham \cite{Mep} depicted in Fig. 
\ref{Fig2}. The term continuation means, we added some values to the puzzle, which were easy 
to find. The cell $10$ (row 2, column 1) is a unicity cell w.r.t. the base set $J$ consisting of the circled 
cells with unique value $4$. This can be seen by the following argumentation. Let $x$ be a solution 
of the reduced problem induced by $J$. In cell $10$ there can be only a $4$ or a $7$. Suppose there 
is a $7$, then $ x_{3,2}=4$, $x_{2,9}=4$ and $x_{7,2}=7$ $\Rightarrow x_{3,9}=7 \Rightarrow 
x_{5,9}=5$ and $x_{8,9}=8 \Rightarrow x_{7,3}=8 \Rightarrow x_{6,3}=3 \Rightarrow 
x_{6,9}\ne 1$ and $x_{6,9} \ne 3$ and this is a contradiction.
\end{example}

This example of Mepham had been analyzed by Crook \cite{Cro}, too. Crook described an algorithm,
which combined the deduction of values with a trial-and-error-method.

Provan \cite{Pro} defined a ``pigeon-hole rule", which is equivalent to the preemptive sets of Crook.
The algorithm of Provan consists of the consecutive application of this rule to a solvable Sudoku puzzle.
If it is possible to apply this rule in each step, then the algorithm leads to a solution and the solution is 
unique \cite[Theorem 2]{Pro}.

\begin{example} \label{E:53}
We consider an example of Provan \cite[Table 2]{Pro} depicted in Fig. \ref{Fig3}. The cell $1$ 
(row 1, column 1) is a unicity cell w.r.t. the base set $J$ consisting of the circled cells with unique 
value $5$. This can be seen by the following argumentation. Let $x$ be a solution of the reduced 
problem induced by $J$. In cell $1$ there can be only a $1$ or $5$. Suppose there is a $1$, then 
$x_{2,2}=2$ and $x_{3,2}=5$ $\Rightarrow x_{8,8}\ne 2 \mbox{ and }\ne 5$. But this is a 
contradiction to the circled values in column 2 and row 8.
\begin{figure}
\includegraphics{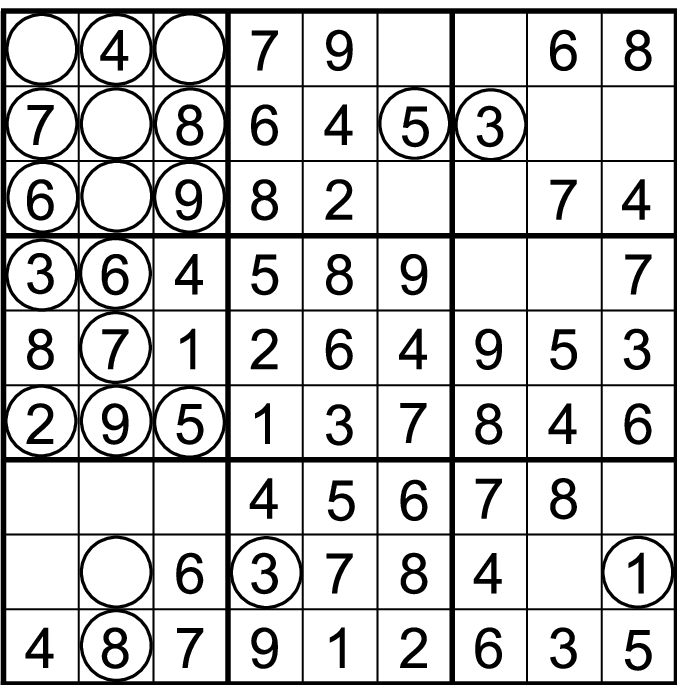}
\caption{The example of Provan} 
\label{Fig3}
\end{figure}
\end{example}

                                        %
                                        %
\section{Unicity by Rectangles} \label{S:unicity_rectangles}
In this section we introduce generalized rectangles, which are sets of cells with a rectangle-type shape
and show the relation to unicity.

\begin{definition}   \label{D:rectangle}
Let $1 \le p \le n$ and $1 \le q \le n$. A set $J \subset \{1, \ldots, n^2\}$ consisting of $p \cdot q$
cells is called a $p$-$q$-rectangle if there exist distinct indices $j_{r,1}, \ldots, j_{r,q}$, $1 \le j_{r,s} \le n$ 
for $s=1, \ldots, q$ and $r=1, 2, 3$, such that 
$\sharp (J \cap cs_{\pi_r}(j_{r,s})) = p$ for $s=1, \ldots, q$ and  $r=1, 2, 3$.
\end{definition}

Each single cell is a $1$-$1$-rectangle. The set of all cells $J = \{1, \ldots, n^2\}$ is an $n$-$n$-rectangle. 
In a classical Sudoku puzzle a set $J$ is a $p$-$q$-rectangle if and only if there exist $q$ rows,
$q$ columns and $q$ blocks and each of them contains $p$ elements of $J$.

\begin{example} \label{E:61}
In Fig. \ref{Fig4} the unpopulated cells form a $2$-$3$-rectangle. The constraint sets are rows 2, 4, 5,
columns 4, 5, 8 and blocks 2, 5, 6. The possible values in this rectangle are $6$ and $7$. We can 
allocate these values in row 2 in two different ways and the values of the remaining 
cells are uniquely determined. In particular this Sudoku puzzle admits two solutions.
\begin{figure}
\includegraphics{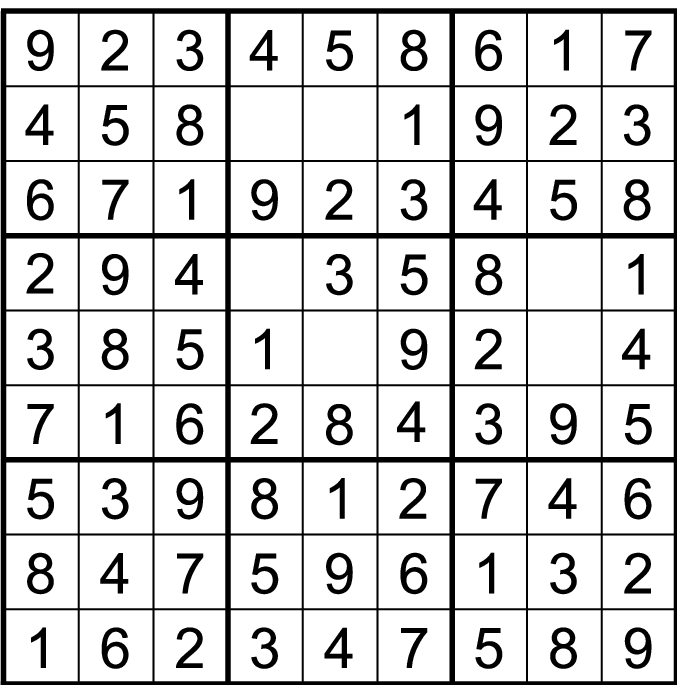}
\caption{A $2$-$3$-rectangle} 
\label{Fig4}
\end{figure}
\end{example}

If $x = (x_1, \ldots, x_{n^2})^T \in S(n, g)$ and $J$ is a $p$-$q$-rectangle, then the $p$ values $x_i$, 
where $i \in J \cap cs_{\pi_1}(j_{1,1})$ are distinct. Consequently, the whole set $\{x_i \mid i \in J\}$ 
contains at least $p$ values.

\begin{definition}
Let $x = (x_1, \ldots, x_{n^2})^T \in S(n, g)$ and let $J$ be a $p$-$q$-rectangle for some $1 \le p \le n$
and $1 \le q \le n$. The point $x$ is called minimal on $J$ if $\sharp \{x_i \mid i \in J\} = p$.
\end{definition}

Assume there exists a solution of the generalized Sudoku problem which is minimal on a $p$-$q$-rectangle 
with $p \ge 2$. Then it is possible to interchange two values in $cs_{\pi_1}(j_{1,s})$ for each $s$ and we 
obtain a new solution, i.e., the generalized Sudoku problem is not uniquely solvable.

\begin{theorem}    \label{T:61}
Let $x \in S(n, g)$ and let $J$ be a $p$-$q$-rectangle for some $2 \le p \le n$ and $1 \le q \le n$. 
Assume $J \cap \{i_1, \ldots, i_k\} = \emptyset$, i.e., $J$ does not contain any given. If $x$ is 
minimal on $J$, the generalized Sudoku problem admits more than one solution.
\end{theorem}
\begin{proof}
Let $x = (x_1, \ldots, x_{n^2})^T$ be a solution of the generalized Sudoku problem, which is minimal 
on the $p$-$q$-rectangle $J$. There exist distinct points $z_1, \ldots, z_p \in \{1, \ldots, n\}$, such that
\[
\{z_1, \ldots, z_p\} = \{x_i \mid i \in J\}= \{x_i \mid i \in J \cap cs_{\pi_r}(j_{r,s})\}
\]
for $s=1, \ldots, q$ and $r=1, 2, 3$. Since $p \ge 2$, there exist two
distinct points $p_1, p_2 \in \{1, \ldots, p\}$. We define
\[
x^\prime_i =
\begin{cases}
z_{p_2}, & \mbox{ if }i \in J \mbox{ and }x_i = z_{p_1}, \\
z_{p_1}, & \mbox{ if }i \in J \mbox{ and }x_i = z_{p_2}, \\
x_i, & \mbox{ if }i \in J, x_i \ne z_{p_1} \mbox{ and }x_i \ne z_{p_2}, \\
x_i, & \mbox{ if }i \in \{1, \ldots, n^2\} \backslash J
\end{cases}
\]
and a new point $x^\prime = (x^\prime_1, \ldots, x^\prime_{n^2})^T \in \mathds{Z}^{n^2}$
with $x^\prime \ne x$. By definition of $x^\prime$ 
\[
\{x_i^\prime \mid i \in J \cap cs_{\pi_r}(j_{r,s}) \} = \{ x_i \mid i \in J \cap cs_{\pi_r}(j_{r,s}) \}
\]
for $s=1, \ldots, q$ and $r=1, 2, 3$. Since $x^\prime_i = x_i$ for $i \notin J$,
\[
\{x_i^\prime \mid i \in cs_{\pi_r}(j) \} = \{x_i \mid i \in cs_{\pi_r}(j) \}
\]
for $j=1, \ldots, n$ and $r=1, 2, 3$ and $x_{i_l}^\prime = g_{i_l}$ for $l=1, \ldots, k$. Using Lemma 
\ref{L:35}, $x^\prime$ solves the generalized Sudoku problem, i.e., we have more than one solution.
\end{proof}

\begin{figure}
\includegraphics{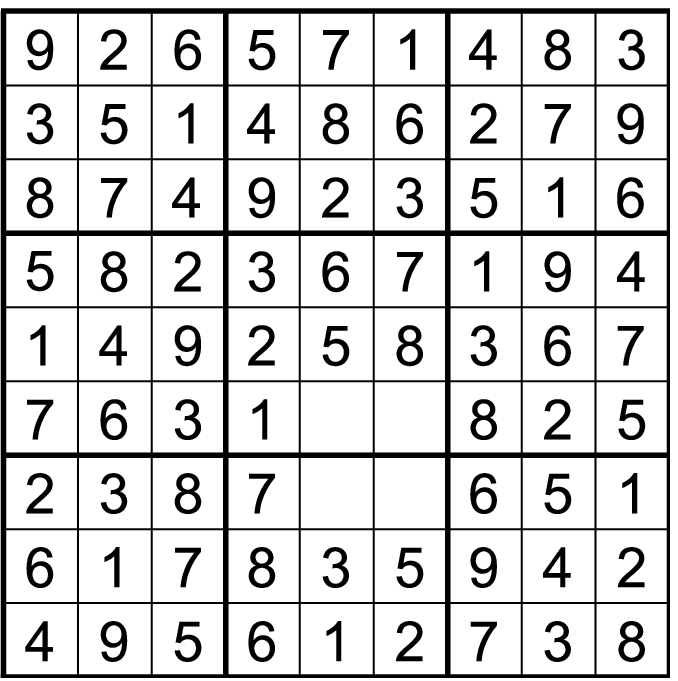}
\caption{The example of Herzberg and Murty} 
\label{Fig5}
\end{figure}

\begin{example} \label{E:62}
We consider an example of Herzberg and Murty \cite[Fig. 4]{HM} depicted in Fig. \ref{Fig5}. The point
$x$ with $(x_{50}, x_{51}, x_{59}, x_{60}) = (9, 4, 4, 9)$ is a solution. The set 
$J = \{50, 51, 59, 60\}$ is a $2$-$2$-rectangle, which does not contain a given and $x$ is minimal on $J$.
By Theorem \ref{T:61}, $x$ is not the only solution. Another solution is $y$ with 
$(y_{50}, y_{51}, y_{59}, y_{60}) = (4, 9, 9, 4)$.
\end{example}

The statement in Theorem \ref{T:61} provides a sufficient condition for multiple solutions. This can be 
reformulated as a necessary condition for unicity.

\begin{theorem}    \label{T:62}
Let $J$ be a $p$-$q$-rectangle for some $2 \le p \le n$ and $1 \le q \le n$. 
Let $x$ be the unique solution of the generalized Sudoku problem, such that $x$ is minimal on $J$.
Then $J$ contains a given, i.e., $J \cap \{i_1, \ldots, i_k\} \ne \emptyset$.
\end{theorem}

The preceding theorem formalizes and generalizes an observation of Herzberg and Murty \cite[p. 712]{HM}, 
who stated ``If in the solution to a Sudoku puzzle, we have a configuration of a type indicated in
Figure 6 in the same vertical stack, then at least one of these entries must be included as a 'given' 
in the initial puzzle, for otherwise, we would have two possible solutions to the initial puzzle simply 
by interchanging $a$ and $b$ in the configuration.". The term ``Figure 6" refers to a figure in their 
paper, which is recovered in Fig. \ref{Fig5}, where the cells $50$, $51$, $59$ and $60$ contain the 
values $a$ respectively $b$.

                                        %
                                        %

\end{document}